\documentclass[10pt]{article}
\usepackage{amsthm}
\usepackage{amsfonts}
\usepackage{amssymb}
\usepackage{amsmath}
\usepackage{nccmath}
\usepackage{mathtools}
\usepackage{mathrsfs}
\usepackage{setspace}
\usepackage{graphicx}
\usepackage[export]{adjustbox}
\usepackage{hyperref}
\usepackage[numbers,sort&compress]{natbib}
\usepackage{enumerate}
\usepackage{authblk}
\usepackage{placeins}

\newtheorem{theorem}{Theorem}[section]
\newtheorem{lemma}[theorem]{Lemma}

\theoremstyle{definition}
\newtheorem{definition}[theorem]{Definition}
\newtheorem{example}[theorem]{Example}

\newtheorem{corollary}[theorem]{Corollary}

\theoremstyle{remark}
\newtheorem{remark}[theorem]{Remark}

\numberwithin{equation}{section}

\usepackage{amsmath,amssymb,graphicx} 
\usepackage[margin=1in]{geometry} 
\usepackage{enumerate}
\usepackage{authblk}
\usepackage{placeins}
\usepackage{array}
\usepackage{academicons}
\usepackage{xcolor}
\usepackage{svg}
\usepackage[utf8]{inputenc}
\usepackage{booktabs}
\usepackage{tikz}
\usepackage{placeins}
\providecommand{\keywords}[1]{\textbf{\textit{Keywords:}} #1}
\providecommand{\subjclass}[1]{\textbf{\textit{MSC2020:}} #1}
\begin{document}

\nocite{*} 

\title{ New Indefinite Summation Formulas and Some Applications }

\author{Hailu Bikila Yadeta \\ email: \href{mailto:haybik@gmail.com}{haybik@gmail.com} }
  \affil{Salale University, College of Natural Sciences, Department of Mathematics\\ Fiche, Oromia, Ethiopia}
\date{\today}
\maketitle

\noindent
\begin{abstract}
\noindent
In this paper, we introduce a novel indefinite summation $\sum_{t} f(t)$ (or antidifference $\Delta ^{-1}f(t) $ ) formula for any given function $f$. We apply the  indefinite summation formula to calculate a particular solution to a nonhomogeneous linear difference equation of the form
$$ y(x+h)-\lambda y(x)=f(x),\quad  h > 0,\quad  \lambda \neq 0, $$
 and also to solve a  linear difference inequality of the form
$$ y(x+h)-\lambda y(x) \geq 0,\quad  h > 0,\quad  \lambda \neq 0.  $$
Furthermore, we apply the formula to determine a particular solution to a difference equations of the form
 $$ \Phi(E)y(t)=f(t), $$
and in solving a linear difference inequality of the form,
 $$ \Phi(E)y(t)\geq 0, $$
 where $   \Phi(E) $ is some linear difference operator. We show how the antidifference of a function $f$ calculated with the current formula is related to the already existing result and establish the corresponding identity.
\end{abstract}

\noindent\keywords{indefinite summation, difference equation, difference operator, antidifference operator, particular solution, floor function, fractional part function, convolution }\\
\\

\section{Introduction}

 According to the superposition principle  of the linear difference equation, the general solution of a nonhomogeneous  linear differential equation is the sum of the general solutions of the associated homogeneous equation and a particular solution of the nonhomogeneous equation. See,for example, \cite{YM}.  The method of finding the particular solution to a nonhomogenous equation requires  the  knowledge of the antidifference (or the indefinite sum) of the nonhomogeneous term. The method of calculation of the antidifference of a given function $f$  may sometimes  be difficult and may involve infinite summation. There are different textbooks with tabulated results for  antidifference of some common functions. See, for example, \cite{LB}. The main purpose of this paper includes the following list:
 \begin{itemize}
   \item To present a new formula for an indefinite summation (or antidifference) of a given function. The current formula is shorter than the ones appearing in the literatures in the sense that it yields the antidifference of a given function $f$ that can be calculated at each point $t$ with a summation with a number of terms not greater than $\lfloor  t \rfloor$ and not an infinite number of terms.
   \item  We apply the current method to the solution of linear difference inequalities as well.
   \item We apply the current result to derive some identities.
   \item We apply the current result to higher order difference operator.
   \item In an analogy to the nonhomogeneous differential equations, where we use the convolution or Green function method, we apply a discrete convolution to a nonhomogeneous difference equations.
   \end{itemize}

\subsection{Periodic and antiperiodic functions}
\begin{definition}
 Let $\mathcal{F} $ represent the space of all real-valued function of real domain. That is
\begin{equation}\label{eq:realvaluedfunction}
   \mathcal{F}:= \{ f: \mathbb{R}\rightarrow \mathbb{R }\}.
\end{equation}
  Let us  denote by $ \mathbb{P}_p $, the space of all periodic functions of period $p$,
  \begin{equation}\label{eq:periodicspace}
   \mathbb{P}_p= \{ f\in \mathcal{F}: f(x+p)=f(x) \},
  \end{equation}
and by $ \mathbb{AP}_p $  the space of all antiperiodic functions of antiperiod $p$.
  \begin{equation}\label{eq:antiperiodicspace}
    \mathbb{AP}_p= \{ f\in \mathcal{F}: f(x+p)= -f(x) \}.
  \end{equation}
\end{definition}
Antiperiodic function with antiperiod $p$ are also periodic function with period $2p$. Periodic and antiperiodic functions are important in the  general solution of linear difference equations with continuous arguments. They  play similar role as the arbitrary consonants play in the indefinite integrals  or the general solution of linear differential equations.
\subsection{The anti difference  operator}
 \begin{definition}
  If $ \Delta F(t) = f(t)  $, then $F(t)$ is called the \emph{antidifference}  of $f(t)$ and is written as $F(t)=  \Delta^{-1}f(t) $.
\end{definition}

Some important properties of difference and antidifferences
\begin{itemize}
  \item $\Delta \Delta ^{-1}f(t)= f(t)$
  \item $ \Delta ^{-1}\Delta f(t)= f(t)+\mu(x), \mu \in \mathbb{P}_1 $.
  \item $\Delta ^{-1} ( \alpha f(t)+\beta g(t))= \alpha \Delta ^{-1}f(t)+\beta \Delta ^{-1}g(t), \alpha, \beta \in\mathbb{ R} $.
  \item If $ \Delta ^{-1}f(t)= F(t),  $ then  $ \Delta ^{-1}f(t+c )= F(t+c), c \in\mathbb{ R } $
  \item If $ \Delta ^{-1}f(t)= F(t)  $ then $ F(t+1)-F(t)= f(t)$.
  \item $ \Delta f(t)\equiv 0  $ if and only if $ f  \in \mathbb{P}_1  $.
    \item Summation by parts:   $ \Delta ^{-1}(u \Delta v) = uv- \Delta ^{-1}(Ev \Delta u )  $.
\end{itemize}
For these and further properties of antidifferences see, for example, \cite{RM},\cite{LB} \cite{CHR},\cite{KM},\cite{SG} \cite{MT}.
\subsection{Some of the available indefinite summation formulas}
There are some antidfference formulas available  for reference and applications. Most of them involve infinite summation, powers of difference operators, derivatives of arbitrary order. Some of the formulas require that the function  whose antidifference is required vanish at infin, some require that the function whose antidifference is required be analytic in some region, and more. Here we look at the most commonly available ones.
\begin{itemize}
  \item \textbf{ Laplace sum}:
   \begin{equation}\label{eq:Laplacesum}
     \sum_{x} f(x) = \int_{0}^{x}f(t)dt -\sum_{k=1}^{\infty}\frac{b_k }{k!}\Delta^k f(x)+c,
   \end{equation}
   where
   $$ b_k= \int_{0}^{1}  \frac{\Gamma (x+1)}{\Gamma( x-k+1)} dx = \int_{0}^{1} (x)_k dx, $$
    are Bernoull's numbers of the second kind, also called  Cauchy numbers of the first kind. See, \cite{GF}, pp. 248, \cite{HG}, pp. 192, \cite{WP}
  \item  \textbf{Newton's formula}:
   \begin{equation}\label{eq:Newtonsfomula}
     \sum_{x} f(x) = \sum_{k=1}^{\infty}\binom{x}{k} \Delta^{k-1}[f](0)+C =\sum_{k=1}^{\infty}\frac{\Delta^{k-1}}{k!}(x)_k+ C
   \end{equation}
   where $(x)_k= \frac{\Gamma (x+1)}{\Gamma (x-k+1)}$ is the falling factorial. See \cite{MT}.
  \item\textbf{ Faulhaber's formula}:
   \begin{equation}\label{eq:Faulhabersformula}
     \sum_{x} f(x) = \sum_{n=1}^{\infty} \frac{f^{n-1}(0)}{n!}B_n(x)+ C
   \end{equation}
   Faulhaber's formula provides that the right-hand side of the equation converges. See \cite{WP}.

  \item   \textbf{Muller's formula}:
If $\lim_{x\rightarrow \infty }f(x)=0$, then
\begin{equation}\label{eq:Mullersformula}
  \sum_{x} f(x)= \sum_{n=0}^{\infty}(f(x)-f(n+x))+C
\end{equation}
See \cite{MS}.
 \item  \textbf{Euler-Maclauren's formula}
\begin{equation}\label{eq:EulerMaclaurensformula}
   \sum_{x} f(x)= \int_{0}^{x}f(t)dt-\frac{1}{2}f(x)+\sum_{k=1}^{\infty}\frac{B_{2k}}{(2k)!}f^{(2k-1)}(x)+C,
\end{equation}
where $B_k$ is the $k$th Bernoulli's number. See \cite{WP}.
\end{itemize}
%

\subsection{Falling factorials and antidifference of polynomials}

\begin{definition}
  The falling factorial
  \begin{equation}\label{eq:fallingfactorial}
    (t)_{n}:= t(t-1)(t-2)...(t-n+1)= \prod_{k=1}^{n}(t-k+1)
  \end{equation}

   The rising factorial
  \begin{equation}\label{eq:risingfactorial}
    t^{(n)}:= t(t+1)(t+2)...(t+n-1)= \prod_{k=1}^{n}(t+k-1)
    \end{equation}
\end{definition}

Powers of  can be written as the linear combinations of falling factorials  as
\begin{equation}\label{eq:ttoninfallingfactorials}
   t^n = \sum_{k=0}^{n} S(n,k)(t)_{k},
\end{equation}
where $S(n,k)$ are Stirling numbers of the second kind. Alternatively, $t^n $ can be written as the linear combinations of rising factorials  as:
 \begin{equation}\label{eq:ttoninrisingfactorials}
   t^n = \sum_{k=0}^{n} (-1)^{n-k}S(n,k)t^{(k)} .
 \end{equation}
The descending  factorials $(t)_{n}$ play a role analogues to that of $t^n $ in differential calculus. That is,
\begin{equation}\label{eq:differenceoffallingfact}
     \Delta (t)_{n} = n (t)_{n-1},
\end{equation}
\begin{equation}\label{eq:antidiffrenceoffallingfact}
 \Delta^{-1}  (t)_{n}=  \frac{(t)_{n+1}}{n+1},
\end{equation}
in analogy to
\begin{equation}\label{derivativexton}
   \frac{d}{dt}t^n = nt^{n-1},
\end{equation}
\begin{equation}\label{eq:integratinzton}
  \int t^n dt = \frac{t^{n+1}}{n+1}.
\end{equation}

But $\Delta ^{-1} t^n $  and $\frac{t^{n+1}}{n+1}$  are not the same.  By linearity of the antidifference operator and by (\ref{eq:ttoninfallingfactorials}),
\begin{equation}\label{eq:antidifferenceofnthpoweroft}
  \Delta^{-1} t^n = \sum_{k=0}^{n}  \frac{  S(n,k) }{n+1}(t)_{k+1}.
\end{equation}
More venereally,  for a polynomial $p_m$ of degree $m$,
$$p_m(t)= \sum_{n=0}^{m} a_n t^n= \sum_{n=0}^{m}\sum_{k=0}^{n} a_n S(n,k) (t)_k $$
We have the antidifference of a polynomial given by
$$ \Delta^{-1} p_m(t)=\sum_{n=0}^{m}\sum_{k=0}^{n} a_n S(n,k)\frac{(t)_{k+1} }{k+1} .  $$
We can calculate the antidifference of a rising factorial as follows:
\begin{equation}\label{eq:antidiffrenceofrisingfactorial}
  \Delta ^{-1}t^{(n)}=  \Delta ^{-1}(E^{n-1}(t)_{n})= E^{n-1}(\Delta ^{-1}(t)_{n})= E^{n-1}\frac{(t)_{n+1}}{n+1}=\frac{(t+n-1)_{n+1}}{n+1}
\end{equation}
Taking into account the relation between the backward difference and the forward difference operators 
\begin{equation}\label{eq:backwardforwardrelation}
  \nabla =I-E^{-1}= E^{-1}( E-I)= E^{-1} \Delta, \quad \nabla^{-1}=  E \Delta^{-1},
\end{equation}
we have the following additional relations:
\begin{equation}\label{eq:backwardantidifferenceoftsubn}
  \nabla^{-1} (t^{(n)})=  E \Delta^{-1}(t^{(n)})= E^n \frac{(t)_{n+1}}{n+1}= \frac{(t+n)_{n+1}}{n+1},
\end{equation}

\begin{equation}\label{eq:backwardantidifferenceoftton}
  \nabla^{-1} (t_{n})=  E \Delta^{-1}(t_{n})= E \frac{(t)_{n+1}}{n+1}= \frac{(t+1)_{n+1}}{n+1}
\end{equation}

\begin{definition}
  The digamma function denoted by $ \Psi $ is the logarithmic derivative of the gamma function:
  $$  \Psi(z):= \frac{d}{dz} [\ln (\Gamma (z))]= \frac{\Gamma'(z)}{\Gamma (z)} $$
\end{definition}
 We have the following
 \begin{equation}\label{eq:theantiddiferenceof1overx}
   \Delta \Psi (t)= \frac{1}{t},\quad \Delta^{-1} \frac{1}{t}= \Psi (t)
 \end{equation}

\begin{theorem}[\textbf{Fundamental Theorem of discrete calculus}]
  Let $F(x)= \Delta^{-1}f(x)$. Then
  $$  \sum_{k=m}^{n} f(k)=\Delta^{-1}f(x)\left.\right|_m^{n+1}=F(m+1)-F(m). $$
\end{theorem}
\begin{example}
  With $f(x)=x^2= x(x-1)+x $, we have
   $$\Delta^{-1}x^2 =\Delta^{-1}x(x-1)+\Delta^{-1} x= \frac{x(x-1)(x-2)}{3}+\frac{x(x-1)}{2}=\frac{x(x-1)(2x-1)}{6} .$$
   Accordingly by fundamental theorem of discrete calculus we have
   $$\sum_{k=1}^{n}k^2=  \left. \frac{x(x-1)(2x-1)}{6}\right\vert^{n+1}_1= \frac{n(n+1)(2n+1)}{6}.$$
\end{example}

\section{Main Results }
\subsection{Particular solution of the linear difference equations  }
 \begin{theorem}\label{eq:Eyeopener1}
 Consider the linear difference equation of the first order
  \begin{equation}\label{eq:1stordnonhomogeneous}
   y(t+1)-\lambda y(t)= f(t),\quad \lambda \neq 0
 \end{equation}
 Then the particular solution of (\ref{eq:1stordnonhomogeneous}) is given by
 \begin{equation}\label{eq:ps1storder}
   y_p(t)=  \sum_{s=1}^{\lfloor t\rfloor }\lambda^{s-1} f(t-s)
 \end{equation}
 \end{theorem}
 \begin{proof}
 Let $y_p(t)$ is defined by (\ref{eq:1stordnonhomogeneous}). Then
 \begin{align*}
    y_p(t+1) =  \sum_{s=1}^{\lfloor t\rfloor +1}\lambda^{s-1} f(t+1-s)&= f(t)+ \sum_{s=2}^{\lfloor t\rfloor +1}\lambda^{s-1} f(t-s+1) \\
    & =f(t)+ \sum_{s=1}^{\lfloor t\rfloor}\lambda^{s} f(t-s) = f(t)+\lambda y_p(t).
 \end{align*}
 \end{proof}
\begin{remark}
The method of finding particular solution in Theorem \ref{eq:Eyeopener1} is equivalent to  calculating the antidifference of the function $f(t)$ given by :
  $$    (E-\lambda I )^{-1} f(t)=  \sum_{s=1}^{\lfloor t\rfloor }\lambda^{s-1} f(t-s).   $$
  In particular for $\lambda = 1 $,
  $$    \Delta ^{-1} f(t)=  \sum_{s=1}^{\lfloor t\rfloor } f(t-s).   $$

\end{remark}

\begin{example}
  Consider the difference equation
  $$ y(t+1)-y(t)=a, a\in \mathbb{R}, a \neq 0  .$$
According to (\ref{eq:ps1storder}), the particular  solution is calculated as
  $$ y_p(t)= \sum_{s=1}^{\lfloor t\rfloor} a = a \lfloor t \rfloor   = a(t -\{t\})= at- a\{t\}.  $$
  Note that for the  fractional part of $t$ which is defined as $\{t\}:= t-\lfloor t \rfloor$, the term $ a\{t\}   \in \mathbb{P}_1 $ and may be absorbed in the general solution of the homogeneous equation. Therefore we may write the particular solution as $ y_p(t)=at $.
  \end{example}

\begin{example}
  Consider the difference equation $y(t+1)-y(t)= t $. According to (\ref{eq:ps1storder}), the particular  solution is calculated as
\begin{align*}
    y_p(t)  =\sum_{s=1}^{\lfloor t \rfloor} (t-s)& = t\lfloor t \rfloor - \sum_{s=1}^{\lfloor t\rfloor} s  \\
     & = t\lfloor t \rfloor - \frac{ \lfloor t \rfloor ( \rfloor t \rfloor + 1)} {2} = \frac{t(t-1)}{2}  - \frac{\{ t \} ( \{ t \}-1 )  }{2}.
  \end{align*}
  We may drop the term $   \frac{\{ t \} ( \{ t \}-1 )  }{2} \in \mathbb{P}_1 $ to only write $ y_p(t)= \frac{t(t-1)}{2} $.
 \end{example}

  \begin{example}
  Consider the difference equation $y(t+1)-y(t)= a^t, a>0, a \neq 1 $. According to (\ref{eq:ps1storder}), the particular  solution is calculated as
$$ y_p(t)  =\sum_{s=1}^{\lfloor t \rfloor} a^{t-s} = a^t   \sum_{s=1}^{\lfloor t\rfloor} a^{ -s }   =  a^t \frac{ \left( a^{-\lfloor t \rfloor  }-1\right) }{1-a} =  \frac{a^t}{a-1} -   \frac{a^{\{t\}}}{a-1}.   $$
Note that for the  fractional part of $t$, which is defined as $\{t\}:= t-\lfloor t \rfloor$, the term $ \frac{a^{\{t\}}}{a-1}  \in \mathbb{P}_1 $ and may be absorbed in the solution of the homogeneous equation. Therefore, we may write the particular solution as $ y_p(t)=\frac{a^t}{a-1}$ or
\begin{equation}\label{eq:exponential}
 \Delta^{-1} a^t =  \frac{a^t}{a-1},\quad  a \neq 1 .
\end{equation}
as it appears in textbooks. See, for example, \cite{LB},\cite{KM}, \cite{MT},\cite{CHR}.
\end{example}

An antidifferences calculated according to (\ref{eq:ps1storder}) differs from the tabulated results by a term that can be absorbed in the solution of the homogeneous equation.


 \begin{theorem}
 Let $F(x):= \Delta ^{-1}f(x) $ is a tabulated antidifference of a function $f(x) $. The antidifference calculated according to (\ref{eq:ps1storder}) and a tabulated result may be related as follows
   \begin{equation}\label{eq:tabulatedvsnew}
       \sum_{s=1}^{\lfloor x \rfloor} f(x-s)= F(x)-F(\{x\})
     \end{equation}
 \end{theorem}
    \begin{proof}
  \begin{align*}
    \sum_{s=1}^{\lfloor x \rfloor} f(x-s)= \left(\sum_{s=1}^{\lfloor x \rfloor}E^{-s} \right)f(x)
    &= \frac{  E^{-1} \left(  E^{-\lfloor x\rfloor}-I\right)}  {E^{-1}-I} f(x)\\
    &= \frac{I- E^{-\lfloor x\rfloor}}{\Delta}f(x)  \\
       & = F(x)-F(\{x\})
    \end{align*}
\end{proof}

\begin{example}
According to tabulated results the antidifference of the function $f(t)=\frac{1}{t}$   is the digamma function $\Psi(t)$. That is
$$\Psi(t) = \Delta^{-1} \frac{1}{t}. $$
By (\ref{eq:tabulatedvsnew}) we have the identity
 \begin{equation}\label{eq:thegammaidentity}
     \Psi (t) = \Psi (\{t\})   +  \sum_{s=0}^{\lfloor t\rfloor } \frac{1}{t-s}  \quad t \notin \mathbb{Z }.
 \end{equation}
 The result in (\ref{eq:thegammaidentity} ) helps us to calculate the digamma function of arbitrary  non-intger  number $t$ once we know the results of digamma function for values in the interval $0 < t< 1$,  by adding a term given by some summation. Similarly, using  (\ref{eq:tabulatedvsnew}) and tabulated result
 $$\Delta ^{-1} \ln t = \ln \Gamma (t), $$ we establish
 \begin{equation}\label{eq:lngammaidentity}
   \sum_{s=1}^{\lfloor t \rfloor}\ln (t-s)= \ln \frac{\Gamma (t) }{\Gamma ( \{t\}) },\quad t \notin \mathbb{Z }.
 \end{equation}
 Equation (\ref{eq:lngammaidentity}) may be written as
 \begin{equation}\label{eq:gammafraction}
   \frac{\Gamma (t) }{\Gamma ( \{t\}) }= \prod_{s=1}^{\lfloor t \rfloor} (t-s),\quad t \notin \mathbb{Z }.
 \end{equation}
\end{example}

\begin{theorem}
  If $f \in \mathbb{P}_T$, then $\Delta ^{-1} f(Tt) = tf(Tt) + \mu (t) $, where $\mu \in \mathbb{P}_1 $.
\end{theorem}
\begin{proof}
 $$\Delta ^{-1} f(Tt)=\sum_{s=1}^{\lfloor t\rfloor}f(T(t-s))=\sum_{s=1}^{\lfloor t\rfloor}f(Tt)=\lfloor t \rfloor f(Tt)=tf(Tt)-\{t\}f (T\{t\}). $$
 The  term $\{t\}f (T\{t\})\in \mathbb{P}_1 $, and can be generalized with an arbitrary term $ \mu(t) \in \mathbb{P}_1  $.
\end{proof}
While the antidifference formula  (\ref{eq:ps1storder}) is we may resort to other techniques for deriving the indifference of other functions.

\begin{example}
Taking the imaginary and the real parts and by using  the formula (\ref{eq:exponential})
  $$ \Delta ^{-1}\sin x = \Delta ^{-1} \Im (e^{ix}) =   \Im  \Delta ^{-1}  (e^{ix})= \Im \frac{e^{ix}}{e^{i}-1}= \frac{\sin(x-1)-\sin x }{2-2 \cos 1 }. $$

  $$ \Delta ^{-1}\cos x = \Delta ^{-1} \Re (e^{ix}) =   \Re  \Delta ^{-1}  (e^{ix})= \Re \frac{e^{ix}}{e^{i}-1}= \frac{\cos(x-1)-\cos x }{2-2 \cos 1 } .$$

\end{example}

\subsection{Particular solution of higher order linear difference equations }

Consider the nonhomogeneous linear difference equation with constant coefficients
  \begin{equation}\label{eq:nonhomogeneous}
    \sum_{i=0}^{n}a_iy(t+i)= f(t), \quad        a_0a_n  \neq 0.
  \end{equation}
  The inhomogeneous equation (\ref{eq:nonhomogeneous} )can be written in operator form as
  \begin{equation}\label{eq:operatorformnthord}
    \Phi(E) y(t)=f(t),
  \end{equation}
  where
  \begin{equation}\label{eqoperator}
    \Phi(E):= \sum_{i=0}^{n}a_iE^{i}, \quad E^0= I .
  \end{equation}
  Let the operator $\Phi(E)y(t) $  can be factored into linear factors as
  \begin{equation}\label{eq:operatorfactor1}
     \Phi(E)=\prod_{i=1}^{n} (E-\lambda_i I)
  \end{equation}
With conditions of the linear difference equation with constant coefficients we formulate the particular solution for the homogeneous equation as in the next theorem.

 \begin{theorem}
 Then the particular solution $y_p$ of  the difference equation $\Phi(E)y(t)= f(t) $ is given by
\begin{equation}\label{eq:lambdamfoldsoln}
   y_p(t)=  \sum_{s_n=1}^{\lfloor t\rfloor }  \sum_{s_{n-1}=1 }^{\lfloor t-s_n\rfloor}...\sum_{s_1=1}^{\lfloor t-s_2-s_3-...-s_{n}\rfloor  }  \left( \prod_{j=1}^{n} \lambda_j^{s_j-1}  \right) f\left( t-\sum_{j=1}^{n} s_j \right) .
 \end{equation}
\end{theorem}

\begin{proof}
  The proof is done by induction over $m$ and the results of Theorem ( ).
\end{proof}

\begin{corollary}
  Let $m \in \mathbb{N}, \lambda \neq 0$. Then the particular solution $y_p$ of the difference equation
\begin{equation}\label{eq:lambdamfolddifference}
(E-\lambda I)^my(t)=f(t)
\end{equation}
  is given by
 \begin{equation}\label{eq:lambdamfoldsoln}
   y_p(t)=  \sum_{s_m=1}^{\lfloor t\rfloor }  \sum_{s_{m-1}=1 }^{\lfloor t-s_m\rfloor}...\sum_{s_1=0}^{\lfloor t-s_2-s_3-...-s_{m}\rfloor  } \lambda^{(s_1+s_2+...+s_m-m)} f\left( t-\sum_{j=1}^{m} s_j \right)
 \end{equation}
\end{corollary}

\subsection{The floor function modulo $h$ and the fraction part function modulo-$h$ }

Similar to the usual floor function and fractional part function, we introduce the floor function modulo $h$ and the fractional part function modulo $h$ as follows:
\begin{definition}
  Let $h > 0 $. We define the floor function modulo $h$ as a function $ \lfloor  . \rfloor_h :  \mathbb{R} \rightarrow \mathbb{Z}  $ by $\lfloor  t \rfloor_h  = n $, where $t= n h +r, 0 \leq r < h $. The fraction part function  modulo $h$  denoted by $ \{ . \}_h$ is defined  as
  $$ \{t \}_h = t- h \lfloor t \rfloor_h. $$
  We also observe that $ \lfloor  x  \rfloor_h  =  \lfloor  \frac{t}{h}  \rfloor  $. But $\{ t \}_h \neq \{ \frac{t}{h} \}$.
\end{definition}
\begin{theorem}
 For a given $h > 0 $, the function $ \lfloor  . \rfloor_h $ is well-defined.
\end{theorem}
\begin{proof}
  Suppose that $t= n_1 h +r_1 = n_2 h + r_2 $ are two possible representations of $t$. Then we have
  $$ |n_1-n_2|h= |r_2-r_1| < h .$$
  This implies that $ |n_1-n_2|< 1 $. Consequently, $|n_1-n_2|=0  $ or $n_1= n_2 $ and $ r_1=r_2 $.
\end{proof}
We have the following result:
\begin{equation}\label{eq:limitsofmoduloh}
  \lim_{h\rightarrow 0 +}  \{ t \}_h = 0,\quad   \lim_{h\rightarrow 0 +} h \lfloor t \rfloor_h = t.
\end{equation}

\begin{lemma}
  $y_p$ is a particular solution of the difference equation $y(x+h)- \lambda y(x) f(x) $ if and only if $y_p$ is the antidifference of the difference operator $E^h-\lambda I $. That is, $y_p(x) = (E^h-\lambda I)^{-1}f(x) $
\end{lemma}

\begin{theorem}
  Let  $y_p$ and $\tilde{y}_p$  be two particular solutions of the first order difference equation
  $$  y(t+h)-y(t)=f(t),$$
  then $y_p (t)=\tilde{y}_p(t)+\mu(t) $ for some $\mu \in \mathbb{P}_h $.
\end{theorem}

 \begin{theorem}
 A particular solution of the linear difference equation
  \begin{equation}\label{eq:hordnonhomogeneous}
   y(t+h)-\lambda y(t)= f(t),\quad \lambda \neq 0,\quad  h >0
 \end{equation}
  is given by
 \begin{equation}\label{eq:pshorder}
   y_p(t)=  \sum_{s=1}^{\lfloor t\rfloor_h }\lambda^{s-1} f(t-hs).
 \end{equation}
 \end{theorem}
 \begin{proof}
 Let $y_p(t)$ is defined by (\ref{eq:pshorder}). Then,
\begin{align*}
    y_p(t+h) =  \sum_{s=1}^{\lfloor t\rfloor_h + 1}\lambda^{s-1} f(t+h-sh)&= f(t)+ \sum_{s=2}^{\lfloor t\rfloor_h +1 }\lambda^{s-1} f(t+h-sh) \\
    & = f(t)+ \sum_{s=1}^{\lfloor t\rfloor_h }\lambda^{s} f(t-sh) = f(t)+\lambda y_p(t).
 \end{align*}
 \end{proof}

 \begin{theorem}
 Let $$ \Phi (E):= \prod_{i=1}^{n} (E^{h_i}-\lambda_i I),\quad h_i >0,\quad \lambda_i \neq 0,\quad  i=1,2,...,n. $$
 Then a particular solution $y_p$ of  the difference equation $\Phi(E)y(t)= f(t) $ is given by
\begin{equation}\label{eq:lambdamfoldsoln}
   y_p(t)=  \sum_{s_n=1}^{\lfloor t\rfloor_{h_n} }  \sum_{s_{n-1}=1 }^{\lfloor t-s_n\rfloor_{h_{n-1}}}...\sum_{s_1=1}^{\lfloor t-s_2-s_3-...-s_{n}\rfloor_{h_1}  }  \left( \prod_{i=1}^{n}. \lambda_i^{s_i-1}  \right) f\left( t-\sum_{j=1}^{n} s_ih_i \right)
 \end{equation}
\end{theorem}

\begin{theorem}
Consider  the difference equation
\begin{equation}\label{eq:mfolddifference}
(E-I)^my(t)=f(t)
\end{equation}
 Then the particular solution of the difference equation (\ref{eq:mfolddifference}) is given by
 $$ y_p(t) = \sum_{s_m=1}^{\lfloor t\rfloor -1}  \sum_{s_{m-1}=0 }^{\lfloor t-s_n\rfloor-2}...\sum_{s_1=0}^{\lfloor t-s_2-s_3-...-s_{n}\rfloor -n }f( t-\sum_{j=1}^{n} s_j-n)    $$
\end{theorem}

\section{ Application to linear  difference inequalities }
\begin{theorem}
  Let $h, \lambda > 0 $. The general solution of the difference inequality
  \begin{equation}\label{eq:horderinequality}
     y(t+h) -\lambda y(t) \geq 0
  \end{equation}
   is given by
  \begin{equation}\label{eq:positivelambdainequality}
    y(t)= \lambda^{\frac{t}{h}}\mu(t) +\sum_{s=1}^{\lfloor t \rfloor_h } \lambda ^{s-1} p(s-sh),
  \end{equation}
  where $p$ any non-negative real valued function and $\mu \in \mathbb{P}_h $.
\end{theorem}
\begin{proof}
  From the inequality $ y(t+h) -\lambda y(t) \geq 0 $, we may establish a nonhomogeneous linear difference equation $ y(t+h) -\lambda y(t) = p(t)  $, where  $p$ is a non-negative function. By superposition principle, the general solution of the nonhomogeneous equation is obtained by sum the solution of the associated homogeneous equation and the particular solution of the nonhomogeneous equation.
\end{proof}

\begin{theorem}
  Let $h>0, \lambda < 0 $. The general solution of the difference inequality $ y(t+h) -\lambda y(t) \geq 0 $ is given by
  \begin{equation}\label{eq:positivelambdainequality}
    y(t)= |\lambda |^{\frac{t}{h}}\mu(t) +\sum_{s=1}^{\lfloor t \rfloor_h }  \lambda ^{s-1}   p(s-sh),
  \end{equation}
  where $p$ any non-negative real valued function and $\mu \in \mathbb{AP}_h $.
\end{theorem}

\begin{theorem}
  Let $h, \lambda > 0 $. The general solution of the difference inequality $ y(t+h) -\lambda y(t) \leq 0 $ is given by

  \begin{equation}\label{eq:positivelambdainequality}
    y(t)= \lambda^{\frac{t}{h}}\mu(t) +\sum_{s=1}^{\lfloor t \rfloor_h }   \lambda ^{s-1} q(s-sh),
  \end{equation}
  where $q$ is any non-positive real-valued function and $\mu \in \mathbb{P}_h $.
\end{theorem}

\begin{theorem}
  Let $h>0, \lambda < 0 $. The general solution of the difference inequality $ y(t+h) -\lambda y(t) \leq 0 $ is given by

  \begin{equation}\label{eq:positivelambdainequality}
    y(t)= |\lambda |^{\frac{t}{h}}\mu(t) +\sum_{s=1}^{\lfloor t \rfloor_h } \lambda ^{s-1}   q(s-sh),
  \end{equation}
  where $q$ is any non-positive real-valued function and $\mu \in \mathbb{AP}_h $.
\end{theorem}

\section{A Convolution representation of antidifference}

In differential equations,  the  solution for some initial value problem for nonhomogeneous equations can be written as a convolution. For example, the solution of the initial value problem for the nonhomogeneous linear difference equation
$$ y'(t)+ \omega y(t)=f(t),\quad y(0)= 0   $$
 is given by the formula
$$   y(t)= \int_{0}^{t}f(s)e^{-\omega(t-s)}ds :=f(t) \ast e^{-\omega t}. $$
\begin{theorem}
For the nonhomogeneous difference equation (\ref{eq:1stordnonhomogeneous}), the particular solution can be written as an antidifference of $f$ given by some convolution as
  $$ y_p(t)= G(t,\lambda) \ast f(t), $$
  where
  $$ G(t;\lambda)= \sum_{s=1}^{\lfloor t\rfloor} \lambda ^{s-1}\delta(t-s)  $$
\end{theorem}
\begin{proof}
\begin{align*}
  G(t;\lambda) \ast f(t)  = \int_{-\infty}^{\infty}G(t-r;\lambda)f(r)dr
   & =\int_{-\infty}^{\infty}   \sum_{s=1}^{\lfloor t\rfloor}\lambda ^{s-1}\delta(t-r-s)f(r)dr \\
   & = \sum_{s=1}^{\lfloor t\rfloor}\lambda^{s-1} \int_{-\infty}^{\infty} \delta(t-r-s)f(r)dr \\
   &= \sum_{s=1}^{\lfloor t\rfloor} \lambda^{s-1}f(t-s)
\end{align*}
\end{proof}

\section{ Discussion of the results and Conclusions}
\begin{itemize}
  \item  A new explicit formula for calculating the antidifference of a given function $f(t)$ is presented in this study. The author is aware that there are various ways to determine the antidifference of a certain function. Some of them focus on common functions such as polynomials, exponentials, and so forth. Additional available formulas are given in terms of infinite series. This requires  knowledge of the interval of convergence of the series as well as the exact value to which the infinite series converges. The current formula is an indefinite sum with  a variable upper limit of summation that is equal to the floor function. This means that, to evaluate the antidifference at a specific point $t$, we have to consider some summation which has $\lfloor t\rfloor$ number of terms. The terms in the summation are just the shift of the given function $f(t)$.  The summation is defined whenever all the terms $f(t-s),\quad  s=1,2,3,..\lfloor  t \rfloor $ are defined .  Therefore tha antidifference  dunction is  evaluated at each ponit  $t$ with a sum with number of terms equal to $\lfloor t \rfloor$ as follows :
$$ \Delta ^{-1}f(t) =\begin{cases}
                        & 0, \quad  \mbox{ if }\quad  -\infty < t <1,  \\
                        & f(t-1), \quad  \mbox{if }\quad  1 \leq t < 2,  \\
                        & f(t-1)+f(t-2),\quad  \mbox{ if } \quad  2 \leq  t < 3,  \\
                        &\vdotswithin{ = } \\[-1ex]
                        & \sum_{s=1}^{n}f(t-s),\quad \mbox{ if }\quad  n \leq t <  n+1.
  \end{cases} $$

 \item  We have considered a forward difference  with positive shifts ($h>0)$. So the interval of calculation for the antidifferece is to the right, while the function is set to zero  for $t < 1 $.. For the negative shift $h<0$, the function may be zero for some interval of the form $[a,\infty ) $ and the antidifference is calculated on the complement of this set.  The case of the back difference operator can be treated with some adjustment. For example, consider the difference equation
    $$\nabla y(t)= y(t)-y(t-1)= f(t)$$
One method is shifting the equation so that it becomes
$$ y(t+1)-y(t) =f(t+1) .$$
Now we can apply the techniques that we have developed for the forward difference operator. Therefore, the backward antidifference operator yields
$$ \nabla ^{-1} f(t) =  \Delta ^{-1} f(t+1) =   \sum_{s=1}^{\lfloor t\rfloor} f(t+1-s) $$
For power function $f(t)= t ^n $, using the notation (\ref{eq:ttoninrisingfactorials}), we may write
$$ \nabla ^{-1} t^n = \nabla ^{-1} \sum_{k=0}^{n} (-1)^{n-k} S(n,k)t^{(k)}= \sum_{k=0}^{n} (-1)^{n-k} \frac{S(n,k)}{k+1}t^{(k+1)}  $$

\item  The current result is applicable in calculating a particular solution of a nonhomogeneous linear difference equations and linear difference inequalities. The indefinite  sums of  some known functions  that appear common textbooks  differ from the result calculated  according to our summation  formula by a term that can be absorbed in the  general solution of the corresponding homogeneous difference equation.
\item The result can be extended to some nonhomogeneous linear difference equations that are not discussed here.
 For example, higher order difference equations with non irreducible difference operator $\Phi(E)$, like $y(t+2)+y(t+1)+y(t)=0$. Some other consideration may yield some summation identities.
  \begin{example}
  Consider the difference equation
$$ y(t+2)-4y(t)=f(t).$$
We can find its particular solution in two different ways: either by taking $h=2$ and $\lambda= 4 $, or by factorising the difference operator $E^2-4I=(E-2I)(E+2I)$ and applying two indefinite summations. Consequently we get the identity:
$$ \sum_{s_2=1}^{\lfloor t \rfloor} \sum_{s_1=1}^{\lfloor t\rfloor- s_2} (-1)^{s_1-1}2^{s_1+s_2}f(t-s_1-s_2)= \sum_{s=1}^{\lfloor t\rfloor_2} 4^sf(t-2s),$$
for any function $f: \mathbb{R}\rightarrow \mathbb{R} $.
  \end{example}

   \begin{example}
  Consider the difference equation
$$ y(t+2)+y(t)=f(t).$$
We can find its particular solution in two different ways: either by taking $h=2$ and $\lambda= -1 $, or by factorising the difference operator $E^2+ I=(E-iI)(E+iI)$ and applying two indefinite summations. Consequently, we arrive at the identity:
$$ \sum_{s_2=1}^{\lfloor t \rfloor} \sum_{s_1=1}^{\lfloor t\rfloor- s_2} (-1)^{s_1}(i)^{s_1+s_2}f(t-s_1-s_2)= \sum_{s=1}^{\lfloor t\rfloor_2} (-1)^{s-1}f(t-2s).$$
for any function $f: \mathbb{R}\rightarrow \mathbb{R} $.
  \end{example}

\item The result can be extended to nonhomogeneous difference equation with the difference operator $\Phi (E)$, where $\Phi$ is a polynomial ineducable over $\mathbb{R}$.
    \begin{example}
    Let
    $$y(t+2)+y(t+1)+y(t)=f(t). $$
Then the particular solution is given by
$$   y(t)= \sum_{s_2=1}^{\lfloor t\rfloor} \sum_{s_1}^{\lfloor t\rfloor-s_2}\omega^{s_1-s_2}f(t-s_1-s_2)  $$
\end{example}

\end{itemize}
\section*{Statements of Declarations:}
\subsection*{Conflict of Interests}
The author declare that there is no conflict of interests regarding the publication of this paper.

\subsection*{Acknowledgment}
The author is thankful to the anonymous reviewers for their constructive and valuable suggestions.
\subsection{Author's contribution} The corresponding author is  the sole contributor of the whole content of this work.
\subsection{Data Availability} There are no external data used in this paper other than the reference materials.

\subsection*{Funding} This Research work is not funded by any institution or individuals.

\end{document}